\documentclass[11pt]{amsart}
\usepackage{amsmath}
\usepackage{amssymb}
\usepackage{amsthm}
\usepackage{latexsym}

\def\NZQ{\mathbb}               

\def\ZZ{{\NZQ Z}}

\def\F2{{\NZQ F}_2}

\def\opn#1#2{\def#1{\operatorname{#2}}} 
\opn\chara{char} \opn\length{\ell} \opn\pd{pd} \opn\rk{rk}
\opn\projdim{proj\,dim} \opn\injdim{inj\,dim} \opn\rank{rank}
\opn\depth{depth} \opn\codepth{codepth} \opn\grade{grade}
\opn\height{height} \opn\embdim{emb\,dim} \opn\codim{codim}

\opn\Tr{Tr} \opn\bigrank{big\,rank}
\opn\superheight{superheight}\opn\lcm{lcm}
\opn\trdeg{tr\,deg}%
\opn\reg{reg} \opn\lreg{lreg} \opn\skel{skel}
\opn\Gr{Gr}
\opn\ann{ann}
\opn\sign{sign}
\opn\del{del}

\opn\div{div} \opn\Div{Div} \opn\cl{cl} \opn\Cl{Cl}
\opn\Spec{Spec} \opn\Supp{Supp} \opn\supp{supp} \opn\Sing{Sing}
\opn\Ass{Ass}\opn\fdepth{fdepth}
\opn\Ann{Ann} \opn\Rad{Rad} \opn\Soc{Soc}
\opn\Sym{Sym} \opn\Ker{Ker} \opn\Coker{Coker} \opn\Im{Im}
\opn\Hom{Hom} \opn\Tor{Tor} \opn\Ext{Ext} \opn\End{End}
\opn\Aut{Aut} \opn\id{id} \opn\ini{in} \opn\tr{tr}

\opn\nat{nat}\opn\it{it}
\opn\pff{proof}
\opn\Pf{proof} \opn\GL{GL} \opn\SL{SL} \opn\mod{mod} \opn\ord{ord}
\opn\aff{aff} \opn\con{conv} \opn\relint{relint} \opn\st{st}
\opn\lk{lk} \opn\cn{cn} \opn\core{core} \opn\vol{vol}
\opn\link{link} \opn\star{star} \opn\skel{skel} \opn\indeg{indeg}
\opn\Ass{Ass} \opn\Min{Min} \opn\sdepth{sdepth} \opn\depth{depth}
\opn\gr{gr}

\def\pot#1#2{#1[\kern-0.28ex[#2]\kern-0.28ex]}

\opn\dirlim{\underrightarrow{\lim}}
\opn\inivlim{\underleftarrow{\lim}}
\def\Implies{\ifmmode\Longrightarrow \else
     \unskip${}\Longrightarrow{}$\ignorespaces\fi}
\def\implies{\ifmmode\Rightarrow \else
     \unskip${}\Rightarrow{}$\ignorespaces\fi}
\def\iff{\ifmmode\Longleftrightarrow \else
     \unskip${}\Longleftrightarrow{}$\ignorespaces\fi}

\let\:=\colon

\let\ol=\overline
\opn\d{d}
\newtheorem{Theorem}{Theorem}[section]
\newtheorem{Lemma}[Theorem]{Lemma}
\newtheorem{Corollary}[Theorem]{Corollary}
\newtheorem{Proposition}[Theorem]{Proposition}

\newtheorem{Example}[Theorem]{Example}

\newtheorem{Definition}[Theorem]{Definition}

\newtheorem{Remark}[Theorem]{Remark}
\let\epsilon\varepsilon
\let\phi=\varphi
\let\kappa=\varkappa
\def\qed{\ifhmode\textqed\fi
   \ifmmode\ifinner\quad\qedsymbol\else\dispqed\fi\fi}
\def\textqed{\unskip\nobreak\penalty50
    \hskip2em\hbox{}\nobreak\hfil\qedsymbol
    \parfillskip=0pt \finalhyphendemerits=0}
\def\dispqed{\rlap{\qquad\qedsymbol}}

\opn\Gin{Gin}

\def\FF{{\mathcal F}}

\def\TT{{\mathcal T}}

\opn\inii{in} \opn\inim{inm} \opn\rate{rate}

\numberwithin{equation}{section}

\title{2-dimensional vertex decomposable circulant graphs}

\usepackage{graphicx}
\usepackage{xcolor}
\usepackage{tikz}
\author{Giancarlo Rinaldo}
\address{Department of Mathematics\\
University of Trento\\
via Sommarive, 14\\
38123 Povo (Trento), Italy
}
\author{Francesco Romeo}
\begin{document}
\maketitle
 \begin{abstract}
 Let $G$ be the circulant graph $C_n(S)$ with $S\subseteq\{ 1,\ldots,\left \lfloor\frac{n}{2}\right \rfloor\}$ and let $\Delta$ be its independence complex. We describe the well-covered circulant graphs with 2-dimensional $\Delta$ and construct an infinite family of vertex-decomposable circulant graphs within this family.
\end{abstract}

\begin{quotation}

\end{quotation}

\section*{Introduction}\label{sec:intro}
Let $n\in \mathbb{N}$ and $S\subseteq\{ 1,2,\ldots,\left \lfloor\frac{n}{2}\right \rfloor\}$. The \textit{circulant graph} $G:=C_n(S)$ is a graph with vertex set $\ZZ_n=\{0,\ldots,n-1\}$ and edge set $E(G) := \{ \{i, j\} \mid |j-i|_n \in S \}$ where $|k|_n=\min\{|k|,n-|k|\}$.

Let $R= K[x_0, \dots, x_{n-1}]$ be the polynomial ring on $n$ variables over a field $K$. The \textit{edge ideal} of $G$, denoted by $I(G)$, is the ideal of $R$ generated by all square-free monomials $x_i x_j$ such that $\{i,j\} \in E(G)$. Edge ideals of graphs have been introduced by Villarreal \cite{Vi} in 1990, where he studied the Cohen--Macaulay property of such ideals. Many authors have focused their attention on such ideals (e.g.\cite{HH}, \cite{MMCRTY}). A known fact about Cohen-Macaulay edge ideals is that they are \textit{well-covered}, that is all the maximal independent sets of $G$ have the same cardinality.
Despite the nice structure the circulant graphs have, it has been proved that is hard to compute their clique number (see \cite{CGV}), and hence the Krull dimension of $R/I(G)$.


In particular, some well-covered circulant graphs have been studied (see \cite{BH0}, \cite{BH1},\cite{MTW}, \cite{EMT} and \cite{Ri}). In \cite{MTW} and \cite{EMT} the authors studied well-covered circulant graphs that are Cohen-Macaulay. The most interesting families are the ones of the power cycle and its complement. In fact, these families contain Cohen-Macaulay edge ideal of Krull dimension $2$. Moreover, the infinite family of well-covered power cycles of Krull dimension $3$ has elements that are all Buchsbaum (\cite{MTW}). 
In addition, in \cite[Table 1]{EMT} the authors studied all the circulant graphs within $16$ vertices, by using a symbolic computation.
Among these, the Cohen-Macaulay ones with Krull dimension $3$ have $R/I(G)$ that is the tensor product of Cohen-Macaulay rings of Krull dimension 1 (e.g. $C_6 (3)$, $C_{9}(3)$, etc.).
We observe that the first non-trivial Cohen-Macaulay circulant graph with Krull dimension 3 is the Paley graph $C_{17}(\ol{1,2,4,8}$) (see Example \ref{P17}). In particular we verified through a \texttt{Macaulay2} computation that its independence complex is also vertex decomposable.
Hence a natural question arises:
``Is it possible to find an infinite family of circulant graphs of Krull dimension $3$ that are Cohen-Macaulay?''

The idea is to find good properties on $n$ and $S$ to find such a family. In fact we will prove the following
\begin{Theorem}\label{2vd}
Let $G=C_{n}(\ol{1,2,4,\ldots,2^m,2^{m}-1})$ with $m\geq 3$ and $n=3\cdot 2^{m}$ and let $\Delta$ be its independence complex. Then $\Delta$ is a $2$-dimensional vertex decomposable simplicial complex with respect to  \[ [1,2,\ldots,\widehat{2}^m,2^{m}+1,\ldots,\widehat{2}^{m+1},\ldots,  n-1].\]
\end{Theorem}

In Section \ref{sec:welc} we give a characterization of pure $2$-dimensional independence complexes of circulants (Proposition \ref{welco}) and explicit formulas for the $f$-vector and $h$-vector (Proposition \ref{2face} and Proposition \ref{hvebu}). In Section \ref{sec:proof}, we give the proof of Theorem \ref{2vd}. Moreover, we present an example, where $n=3\cdot 2^3$, to clarify the steps of the proof of Theorem \ref{2vd} (Example \ref{exam}).
Furthermore in Section \ref{sec:levG} we prove that any 2-dimensional vertex decomposable independence complex of circulants has Stanley-Reisner ring that is a level algebra (Theorem \ref{typ}). It is known that the Hilbert function of level algebras has nice properties (see \cite{GHMS}). Moreover when one talks about level algebras, the question about which ones are also Gorenstein algebras (see \cite{BH2}, \cite{GHMS} ) naturally arises. In this regard, in Proposition \ref{Gor} we prove that among the level algebras of Theorem \ref{typ}, the only Gorenstein algebra is $R/I(G)$ where $G=C_{6}(3)$.

\section{Preliminaries and the Paley example}
In this section we recall some concepts and notations on graphs and on simplicial complexes that we will use in the article. 

\noindent Set $V = \{x_1, \ldots, x_n\}$. A \textit{simplicial complex} $\Delta$ on the vertex set $V$ is a collection of subsets of $V$ such that: 1) $\{x_i\} \in \Delta$ for all $x_i \in V$; 2) $F \in \Delta$ and $G\subseteq F$ imply $G \in \Delta$.
An element $F \in \Delta$ is called a \textit{face} of $\Delta$. A maximal face of $\Delta$ with respect to inclusion is called a \textit{facet} of $\Delta$.

\noindent
The dimension of a face $F \in \Delta$ is $\dim F = |F|-1$, and the dimension of $\Delta$ is the maximum of the dimensions of all facets.
Moreover, if all the facets of $\Delta$ have the same dimension, then we say that $\Delta$ is \emph{pure}. Let $d-1$ the dimension of $\Delta$ and let $f_i$ be the number of faces of $\Delta$ of dimension $i$ with the convention that $f_{-1}=1$. Then the $f$-vector of $\Delta$ is the $(d+1)$-tuple $f(\Delta)=(f_{-1},f_0,\ldots,f_{d-1})$. The $h$-vector of $\Delta$ is $h(\Delta)=(h_0,h_1,\ldots,h_d)$ with
\begin{equation} \label{hve}
 h_k=\sum_{i=0}^{k}(-1)^{k-i}\binom{d-i}{k-i} f_{i-1}. 
\end{equation}
The sum
\[
 \widetilde{\chi}(\Delta)=\sum_{i=-1}^{d-1}(-1)^{i}f_i
\]
is called the \emph{reduced Euler characteristic} of $\Delta$ and $h_d=(-1)^{d-1}\widetilde{\chi}(\Delta)$.
For any $F \in \Delta$  we define $\link_{\Delta}(F)=\{ G \in \Delta : \ F \cap G= \varnothing \ \mbox{ and } F \cup G \in \Delta \}, \ \del_{\Delta}(F)=\{G \in \Delta : F\cap G =\varnothing \}.$\\
We define the \emph{chain complex} as follows:
\[
\mathcal{C} \ : \ 0 \rightarrow K^{f_{d-1}} \overset{\partial_{d-1}}{\longrightarrow} K^{f_{d-2}} \overset{\partial_{d-2}}{\longrightarrow} \ldots \overset{\partial_{0} }{\longrightarrow} K \rightarrow 0
\]
and by definition the $i-$th \emph{reduced homology group} $\widetilde{H}_{i}(\Delta;K)$ is
\[
\widetilde{H}_{i}(\Delta;K)= \ker (\partial_{i})/ \mbox{im} ({\partial_{i+1}}).
\]
We set $b_{i}=\dim \widetilde{H}_{i}(\Delta;K)$ and we point out that 
\begin{equation}\label{b-1}
b_{-1} = 1 \Leftrightarrow \Delta= \{\varnothing\};
\end{equation}
\begin{equation}\label{b0}
b_{0}= c-1
\end{equation}
where $c$ is the number of distinct components of $\Delta$.
As in \cite[Chapter 7]{Hi}, the reduced Euler characteristic $\widetilde{\chi}(\Delta)$ can be seen as
\begin{equation}\label{homchi}
\widetilde{\chi}(\Delta)=\sum\limits_{i = -1}^{d-1} (-1)^{i} b_{i}= (-1)^{d-1} h_{d}.
\end{equation}

\noindent Given any simplicial complex $\Delta$ on $V$, we can associate a monomial ideal $I_\Delta$ in the polynomial ring $R$ as follows:
\[
 I_\Delta=(\{x_{j_1}x_{j_2}\cdots x_{j_r}: \{x_{j_1},x_{j_2},\ldots,x_{j_r}\}\notin \Delta\}).
\]
$R/I_\Delta$ is called Stanley-Reisner ring and its Krull dimension is $d$. If $G$ is a graph we call the \textit{independence complex} of $G$ by
\[
\Delta(G)=\{A\subset V(G): A \mbox{ is an independent set of }G\}. 
\]

The \textit{clique complex} of a graph $G$ is the simplicial complex whose faces are the cliques of $G$.

\noindent Let $T=\{1,2,\ldots,\left \lfloor\frac{n}{2}\right \rfloor\}$ and $G$ be a circulant graph on $S \subseteq T$. We observe that $\ol{G}$ is a circulant graph on $\ol{S}=T\setminus S$ and the clique complex of $\ol{G}$ is the independence complex of $G$, $\Delta(G)$. So from now on we will take $\Delta$ as the clique complex of the graph $\ol{G}=C_{n}(\ol{S})$.

Let $\Delta$ be a pure independence complex of a graph $G$.
We say that $\Delta$ is \emph{vertex decomposable} if one of the following conditions hold:
(1) $n=0$ and $\Delta=\{\varnothing\}$; (2)  $\Delta$ has a unique maximal facet $\{x_{0},\ldots,x_{n-1} \}$;  (3) There exists $x \in V(G)$ such that both $\link_{\Delta}(x)$ and $\del_{\Delta}(x)$ are vertex decomposable and the facets of $\del_{\Delta}(x)$ are also facets in $\Delta$.\\
We say that $\Delta$ is \emph{Cohen-Macaulay} if for any $F \in \Delta$ we have that \\$\dim_K \widetilde{H}_{i}(\link_{\Delta}(F),K)=0$ for any $i < \dim \link_{\Delta}(F)$.		\\
We say that $\Delta$ is \emph{Buchsbaum} if $\forall \{x\} \in \Delta$ we have that $\link_{\Delta}(x)$ is Cohen-Macaulay.

It is well known that 
\[
\Delta \ \mbox{Vertex Decomposable} \Rightarrow \] \[ \Rightarrow \Delta \ \mbox{Cohen-Macaulay} \ \Rightarrow \ \Delta \  \mbox{Buchsbaum} \Rightarrow \Delta \ \mbox{Pure}.
\]
\begin{Remark}\label{isove}
Let $\Delta$ be a $0$-dimensional simplicial complex on $n$ vertices. Then $\Delta$ is vertex decomposable.
\end{Remark}

\begin{Lemma}\label{1dim}
Let $\Delta$ be a $1$-dimensional simplicial complex on $n$ vertices. Then the following are equivalent
\begin{itemize}
\item[(i)]  $\Delta$ is vertex decomposable;
\item[(ii)] $\Delta$ is connected.
\end{itemize}
\end{Lemma}

Let $\mathbb{F}$ be the minimal free resolution of $R/I(G)$. Then 
\[
\mathbb{F} \ : \ 0 \rightarrow F_{p} \rightarrow F_{p-1} \rightarrow \ldots \rightarrow F_{0} \rightarrow R/I(G)\rightarrow 0
\]
where $F_{i}=\bigoplus\limits_j R(-j)^{\beta_{i,j}}$. The $\beta_{i,j}$ are called the \emph{Betti numbers} of $\mathbb{F}$.
For any $i$, $\beta_{i}=\sum_{j} \beta_{i,j}$ is called the $i$-th \emph{total Betti number}.
The \emph{Castelnuovo-Mumford regularity} of $R/I(G)$, denoted by $\mbox{reg}\  R/I(G)$ is defined as 
\[
\mbox{reg} \ R/I(G)=\max\{j-i: \beta_{i,j} \}.
\]
Let $\sigma \subseteq V= \{x_{0},\ldots ,\ x_{n-1} \}$. We define the \textit{restriction} of the simplicial complex $\Delta$ to $\sigma$ as
\[
\Delta |_{\sigma}= \{F \in \Delta \ | \ F \subseteq \sigma  \}.
\]

\begin{Theorem}[Hochster's formula, \cite{MS}]\label{hoch}
The non-zero Betti numbers of $R/I_{\Delta}$ lie in the squarefree degree $j$, and we have
\[
\beta_{i,j}(R/I_{\Delta})=\sum_{|\sigma|=j, \ \sigma\subseteq V}\dim_{K} \widetilde{H}_{j-i-1}(\Delta|_{\sigma};K).
\]
\end{Theorem} 

If $R/I(G)$ is Cohen-Macaulay, then the last total Betti number $\beta_p$ is the \textit{Cohen-Macaulay type}  of $R/I(G)$. 
Moreover, if the Cohen-Macaulay type is $\beta_{p,p+\reg R/I(G)}$, then $R/I(G)$ is called a \emph{level algebra}. When the Cohen-Macaulay type is equal to 1, we say that $R/I(G)$ is a \textit{Gorenstein algebra}.

\begin{Remark}\label{leval}
Let $G$ be a Cohen-Macaulay graph with independence complex $\Delta$ such that $\widetilde{\chi}(\Delta)\neq 0$ and Cohen Macaulay type $s$. Then $R/I(G)$ is a level algebra if and only if $s=|\widetilde{\chi}(\Delta)|$. 
\end{Remark}
\begin{proof}
We have to compute $\beta_{p,p+r}$, where $p=\pd R/I(G)$ and $r=\reg R/I(G)$. Let $d$ be the Krull dimension of $R/I(G)$. Since $R/I(G)$ is Cohen-Macaulay and from Auslander-Buchsbaum formula, we have that $\pd R/I(G)=n-d$. Moreover, since $\widetilde{\chi}(\Delta)\neq 0$ from \cite[Remark 1.2]{RR} and \cite[Corollary 4.8]{Ei}, $\reg R/I(G)=\depth R/I(G)= d$.
Hence, 
\[
\beta_{n-d,n}\overset{(*)}{=}\dim \widetilde{H}_{d-1}(\Delta;K)\overset{(**)}{=}|\widetilde{\chi}(\Delta)|,
\]
where $(*)$ follows from Theorem \ref{hoch} and $(**)$ from the fact that $\Delta$ is Cohen-Macaulay and from Equation \eqref{homchi}. Then the assertion follows.
\end{proof}

We now give an example of circulant graph whose independence complex is vertex decomposable.
\begin{Example}\label{P17}
Let us consider the circulant graph $G=C_{17}(\ol{1,2,4,8})$, that is the the Paley 17. (see \cite{Bo}). In Figure \ref{pal} we represent $\link_\Delta (x)$ for an $x \in V(G)$.
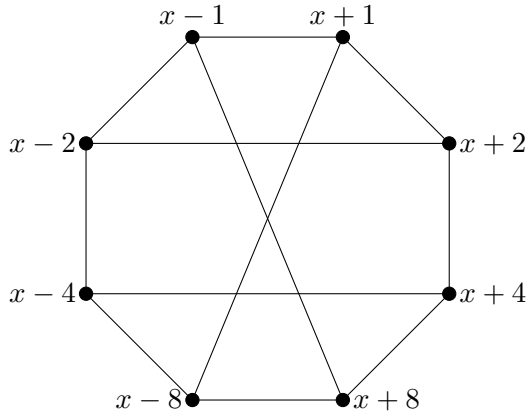
\begin{figure}[h]
\begin{center}
\begin{tikzpicture}
\draw (-1,1) -- (1,1);
\draw (-2.4142135,-2.4142135) --(2.4142135,-2.4142135);
\draw (1,1) -- (2.4142135,-0.4142135);
\draw (2.4142135,-2.4142135) -- (2.4142135,-0.4142135);
\draw (-2.4142135,-0.4142135) -- (2.4142135,-0.4142135);
\draw (-2.4142135,-0.4142135)--(-2.4142135,-2.4142135);
\draw (-2.4142135,-0.4142135)--(-1,1);
\draw (-1,-3.8284271)--(-2.4142135,-2.4142135);
\draw (2.4142135,-2.4142135) -- (1,-3.8284271);
\draw (-1,-3.8284271) -- (1,-3.8284271);
\draw (1,-3.8284271) -- (-1,1);
\draw (-1,-3.8284271) -- (1,1);
\filldraw (-1,-3.8284271) circle (2.5 pt) node [anchor=east] {$x-8$};
\filldraw (1,-3.8284271) circle (2.5 pt) node [anchor=west] {$x+8$};
\filldraw (-2.4142135,-0.4142135) circle (2.5 pt) node [anchor=east] {$x-2$};
\filldraw (2.4142135,-0.4142135) circle (2.5 pt) node [anchor=west] {$x+2$};
\filldraw (1,1) circle (2.5 pt) node [anchor=south] {$x+1$};
\filldraw (-1,1) circle (2.5 pt) node [anchor=south] {$x-1$};
\filldraw (2.4142135,-2.4142135) circle (2.5 pt) node [anchor=west] {$x+4$};
\filldraw (-2.4142135,-2.4142135)circle (2.5 pt) node [anchor=east] {$x-4$};
\end{tikzpicture}\caption{$\link_{\Delta}(x)$ for any $x\in V(C_{17}(\ol{1,2,4,8}))$}\label{pal}
\end{center}
\end{figure}

This graph is the first example of Cohen-Macaulay circulant graph whose Stanley-Reisner ring has Krull dimension 3 and it is not a tensor product of rings having Krull dimension 1. Moreover, it is a level algebra.
\end{Example}

\section{2-dimensional well-covered independence complexes}\label{sec:welc}
We start by providing a description of well-covered graphs of Krull dimension $3$ in terms of the elements in $S$. 
\begin{Proposition}\label{welco}
Let $G=C_{n}(S)$ be a non-complete circulant graph. Then $\Delta$ is a pure simplicial complex of $\dim \Delta=2$ if and only if for any $a\in \ol{S}$ 
the following conditions hold
\begin{enumerate}
\item There exists $|b|_n\in\ol{S}$ such that $|b-a|_n \in \ol{S}$;
\item For any $|b|_n,|c|_n\in\ol{S}$ with $b\neq c$ and such that $|c-a|_n,|b-a|_n \in \ol{S}$ then $|b-c|_n \notin \ol{S}$. 
\end{enumerate}
\end{Proposition}
\begin{proof}
$\Rightarrow)$ 
Let $a$ be such that $|a|_n \in \ol{S}$. Suppose by contraposition that at least one of the conditions (1) or (2) does not hold. Firstly, let us suppose that there does not exist $b \in V(G)$ such that $|b|_n, |b-a|_n \in \ol{S}$, then the edge $\{0,a\}$ is a facet of dimension $1$ of $\Delta$. It contradicts the assumption. Secondly, if there exist $ b,c \in V(G)$ such that $|b|_n,|c|_n,|c-a|_n,|b-a|_n \in \ol{S}$ and $|b-c|_n \in \ol{S}$ then $\{0,a,b,c\}$ is a facet of dimension $3$ of $\Delta$. It contradicts the assumption. \\
$\Leftarrow)$ 
We start by proving that $\Delta$ is $2$-dimensional. From $(1)$ it follows that $\{0,a,b\} \in \Delta$, namely $\dim \Delta \geq 2$. Now we prove $\dim \Delta \leq 2$.
By contraposition let $\dim \Delta >2$, then there exists a facet of dimension $3$, namely $\{0,a,b,c\}$ so that 
\[
|a|_n,|b|_n,|c|_{n}, |c-a|_n, |b-a|_n, |b-c|_n \in \ol{S}
\]
that contradicts (2).\\ Now we prove that $\Delta$ is pure. By contraposition, assume $\Delta$ is $2$-dimensional but not pure. Since $G$ is not complete, $\bar{S}$ is not empty, namely there are no isolated vertices. Then there exists $a\in V(G)$ such that $\{0,a\}$ is a facet of $\Delta$, and in particular $|a|_n \in \bar{S}$. It contradicts the assumption (1).
\end{proof}
Now we prove some properties on the $f$-vector and the $h$-vector of $2$-dimensional independence complexes of circulants.

\begin{Proposition}\label{2face}.
Let $G=C_{n}(S)$ be such that $\dim \Delta=2$ and
\[
\mathcal{F}_{0}=\Big\{\{0,a,b\} \subset V : |a|_n ,|b|_n, |b-a|_n\in \bar{S}  \Big\},
\]
the set of the 2-dimensional facets of $\Delta$ containing the vertex $0$ and let
\[
\FF_0=\TT\sqcup \TT_e
\]
where 
\[
\TT_e=\Big\{\{0,a,b\} \subset V : |a|_n =|b|_n=|b-a|_n\in \bar{S}  \Big\}.
\]
Then $|\TT|=3t$, for some $t \in \mathbb{N}$ and
\[
|\TT_{e}|=\begin{cases}1& \mbox{if } n=3k \mbox{ with } k \in \ol{S}\\ 0 &\mbox{otherwise.}\end{cases}
\]
\end{Proposition}
\begin{proof}
For any $F=\{0,a,b\} $ in $\mathcal{T}$ by shifting the elements of $F$ by $a$ and $b$, we obtain the sets $F(-a)=\{-a,0,b-a\},\ F(-b)=\{-b,a-b,0\}$ that are distinct and belong to $\TT$. These are the only shifts sending $F\in \FF_0$ in $F'\in \FF_0$. Hence
\[
3 \ | \ |\mathcal{T}|.
\]
By similar argument if $F\in \TT_e$ through the shifts we obtain $F$ itself.
\end{proof}

\begin{Remark}\label{no2k}
We highlight that there is no circulant $C_{2k}(S)$ such that $k \in \ol{S}$ with pure and 2-dimensional $\Delta$. In fact by contraposition let us assume such a $G$ exists. Since $\Delta$ is pure, there exists at least an $a$ such that $\{0,a,k\}$ is a $2$-face of $\Delta$, with $|k-a|_{n} \in \ol{S}$. If we set $a=a$, $b=k$ and $c=a+k$ we have that $|b|_{n}, |c|_{n}=|k-a|_{n}, |c-a|_{n}=|k|_{n}, |b-a|_{n}, |b-c|_{n} \in \ol{S}$ that contradicts (2) of Proposition \ref{welco}.
\end{Remark}
\begin{Proposition}\label{hvebu}
Let $G=C_{n}(S)$ be such that $\Delta$ is a pure simplicial complex of dimension $2$. Then $h(\Delta)=(1,\ n-3, \ n(s-2)+3, \ h_3)$ with
\begin{equation}\label{h3}
h_3=
\left \{
  \begin{tabular}{ll}
  $-1+n(t-s+1)+ k$ & if  $n=3k$ with $k \in \ol{S}$\\
  $-1+n(t-s+1)$    & otherwise  
  \end{tabular}
\right. 
\end{equation}
where $s=|\ol{S}|$ and $t=\frac{1}{3} |\mathcal{T}|$. 
\end{Proposition}
\begin{proof}
By plugging $k=1,2$ and $d=3$ in formula \eqref{hve}, we obtain
\[
h_{1}=f_{0}-3; \ \ h_{2}=f_{1}-2n+3.
\] 
Since from Remark \ref{no2k} if $n$ is even $\frac{n}{2} \notin \bar{S}$, we have $f_{1}=ns$. Moreover,
\[
h_3=\widetilde{\chi}(\Delta)=-1+f_{0}-f_1+f_{2}=-1+n-ns+f_{2}.
\]
From \cite[Lemma 1]{Ri} we have $f_{2}=\frac{n \cdot  |\mathcal{F}_0| }{3}$, where  $|\mathcal{F}_0|$ is the number of 2-dimensional facets of $\Delta$ containing the vertex $0$. 
By notation of Proposition \ref{2face}, we set
\[
t:= \frac{1}{3} |\mathcal{T}|.
\]
If $n \neq 3k$ or $n=3k$ with $k\notin \bar{S}$ we have $|\mathcal{T}_e|=0$ and $ |\mathcal{F}_0|=3t$.
In the case $n=3k$ and $k \in \ol{S}$ we have $ |\mathcal{F}_0|=3t+1$, that yields $f_{2}=\frac{n (3t+1)}{3}=nt+k$. Hence \eqref{h3} follows. 
\end{proof}

\section{Proof of Theorem \ref{2vd}}\label{sec:proof}
The aim of this section is to prove Theorem \ref{2vd}. We first prove that $\Delta$ is pure and 2-dimensional, computing its $f$-vector.

\begin{Proposition}\label{2wel}
Let $G=C_{n}(\ol{1,2,4,\ldots,2^m,2^{m}-1})$, $m\geq 3$ and $n=3 \cdot 2^m$.
Then $\Delta $ is a pure $2$-dimensional simplicial complex with $f$-vector 
\[
  (1,n,n(m+2),n(m+2)+2^m)
 \]
\end{Proposition}
\begin{proof}
We prove that $\Delta$ is a pure 2-dimensional simplicial complex, by using Proposition \ref{welco}. For this aim, we describe the $2$-faces of $\Delta$ containing the vertex $0$.
In the notation of Proposition \ref{2face}, $\{0,2^m,2^{m+1}\} \in \mathcal{T}_{e}$ and $\mathcal{T}$ is formed by the elements $F\in\{\{0,a,b\}: a,b \in \bar{S}  \}$ that are 
\begin{equation}\label{T}
\{0,2^{i},2^{i+1}\}_{i=0,\ldots,m-1}, \{0,1,2^{m}\} ,   \{0,2^{m}-1,2^{m}\}, 
\end{equation}
and their shifts $F(-a)=\{-a,0,b-a\},\ F(-b)=\{-b,a-b,0\}$.
Therefore for any $a\ \in \bar{S}$, $\{0,a\}$ is not a facet of $\Delta$,  condition $(1)$ of Proposition \ref{welco}. 

To verify condition $(2)$, we claim that there are no faces $\{0,a,b,c\}\in \Delta$. To prove this claim we distinguish two cases:
\begin{enumerate}
 \item[(C1)] $a,b,c \in \bar{S}$;
 \item[(C2)] $a,b \in \bar{S}$, with $a<b$ and $c \in \{-s \ : \ s \in \bar{S} \}$.
\end{enumerate}
By symmetry the other cases follow.

\noindent

(C1) We need to verify  that for all $\{0,a,b\}$ and $\{0,a,c\}$ in \eqref{T} we have $|b-c|\notin \bar{S}$. For any $i \in \{1,2,\ldots, m-2\}$  we have $|2^{i+2}-2^i|_n= 2^{i+2}-2^i \notin \bar{S}$, then $\{0,2^i,2^{i+1},2^{i+2}\}$ is not a $3$-face of $\Delta$. Furthermore, $|2^{m}-2|=2^m-2\notin \bar{S}$ because $m\geq 3$, that is $\{0,1,2^{m}-1,2^{m}\}$, $\{0,1,2,2^{m}\}$ are not $3$-faces of $\Delta$. By similar arguments, the remaining cases follow.\\
(C2) The strategy is the following. We consider the vertices that are adjacent to both  $0$ and $a$ and prove that within this set each pairs of candidates satisfying (C2) are not in $\Delta$.

Let $a=1$. We observe that the vertices adjacent to $0$ and $1$ are $\{2,-1,2^m,1-2^m\}$. The candidates  $\{b,c\}$ are \[\{2,-1\},\{2,-2^{m}+1\},\{2^m, -1\},\{2^m ,-2^{m}+1\}.\] It is straightforward to see that the above pairs are not in $\Delta$.

Let $a=2^{i}$ with $1\leq i \leq m-1$. Then the only candidate $\{b,c\}$ is $\{2^{i+1},-2^{i}\}$. The latter is not in $\Delta$.


Let $a=2^{m}-1$, then $b=2^m$ and the only candidate for $c \in \{-s \ : \ s \in \bar{S} \}$ is $-$1. But $\{2^m,-1\}$ is not in $\Delta$.


Hence we have $\Delta$ is pure and $2$-dimensional. 
For what matters the $f$-vector of $\Delta$, we have that $f_{1}=ns$ where $s=|\bar{S}|$, hence $s=m+2$. According to Proposition \ref{2face} and \cite[Lemma 1]{Ri} we have $f_{2}=\frac{n \cdot |\mathcal{F}_0|}{3}$. The elements of  $\mathcal{F}_{0}$ are  the $2$-faces  $F$ in \eqref{T} and the shifted ones plus the one of $\mathcal{T}_e$. Hence, they are $3(m+2)+1$, and $f_{2}=\frac{3n(m+2)+n}{3}=n(m+2)+2^m$.
\end{proof}

We present a characterization of vertex decomposability for $2$-dimensional simplicial complexes useful for our aim.

\begin{Lemma}\label{weldel}
Let $\Delta$ be a $2$-dimensional pure connected simplicial complex on $n$ vertices, let $\mathcal{M}=\{v_1,v_{2},\ldots, v_{n-3}\}$ be a sequence of vertices of $\Delta$, and for $i=1,2,\ldots, n-3$ let
\[
\Delta_{i-1}=\begin{cases} \Delta &\mbox{if } i=1 \\ \del_{\Delta_{i-2}}(v_{i-1}) &\mbox{otherwise}. \end{cases}
\]
Then the following are equivalent:
\begin{itemize}
\item[(i)] $\Delta$ is vertex decomposable with respect to $\mathcal{M}$;
\item[(ii)] $\mathcal{M}$ satisfies the following properties: 
\begin{enumerate}
\item For any $i=1,2,\ldots, n-3$, $\link_{\Delta_{i-1}}(v_i)$ is a connected 1-dimensional simplicial complex;
\item $\Delta_{n-3}$ is the simplex on $3$ vertices.
\end{enumerate}
\end{itemize}
\end{Lemma}
\begin{proof}
(i) $\Rightarrow$ (ii). By contraposition, we assume that one of the following is true:
\begin{enumerate}
\item[(1)'] There exists a $k\in \{1,2,\ldots, n-3\}$ such that $\link_{\Delta_{k-1}}(v_k)$ is disconnected or 0-dimensional;
\item[(2)'] $\Delta_{n-3}$ is not the simplex on $3$ vertices.
\end{enumerate}
If (1)' and there exists a $k\in \{1,2,\ldots, n-3\}$ such that $\link_{\Delta_{k-1}}(v_k)$ is a disconnected and $1$-dimensional, then $\link_{\Delta_{k-1}}(v_k)$ is not vertex decomposable according to Lemma \ref{1dim}, hence $\Delta$ is not vertex decomposable. If it is $0$-dimensional, then there exists an isolated vertex $b$ in $\link_{\Delta_{k-1}}(v_k)$, that is $\{v_{k},b\}\in \mathcal{F}(\Delta_{k-1})$ and the facets of $\Delta_{k-1}$ are not facets of $\Delta$, that contradicts the assumption of vertex decomposability.
If (2)', then $\dim \Delta_{n-3} < 2$ and so the facets of $\Delta_{n-3}$ are not facets in $\Delta$.\\
(ii) $\Rightarrow$ (i).
We claim the sequence $\mathcal{M}$ is a sequence that is a vertex decomposition of $\Delta$.
From Lemma \ref{1dim} and the property (ii).($1$) we obtain that $\link_{\Delta_{i-1}}(v_i)$ for $i=1,2,\ldots, n-3$ are vertex decomposable.  Hence to prove that $\Delta$ is vertex decomposable we are left with proving the following\\
\textbf{Claim:  }$\mbox{For any } i=1,\ldots,n-3 \mbox{ the facets of } \Delta_i  \mbox{ are facets of } \Delta$.\\
Let us assume condition (ii).($2$) and that there exist a $j\in \{1,2,\ldots,n-3\}$ and $\{a,b\}$ is a facet of $\Delta_j$. Then at least one between $a$, and $b$ leaves in $\mathcal{M}$. In fact, if both $a,b$ do not live in $\mathcal{M}$, then $\{a,b\}$ will be an edge of $\Delta_{n-3}$, that we recall is a simplex on $3$ vertices. That is impossible.
So let us assume that there exists a $k>i$ such that $a=v_k$ and $\{v_{k},b \} \in \mathcal{F}(\Delta_j) $. It implies that $\link_{\Delta_{k-1}}(v_k)$ contains $b$ as isolated vertex, that contradicts the property (ii).($1$). 
Hence the claim follows.
\end{proof}

Now we prove the main theorem.

\begin{proof}[Proof of Theorem \ref{2vd}]
From Proposition \ref{2wel}, $\Delta$ is pure and $2$-dimensional.
To prove the vertex decomposability of $\Delta$, it is useful to define the following edge sets
\[
E(\mathcal{H}_{v}^{l})=\Big\{ \{v-2^{i},v-2^{i+1}\}_{i=l,l+1,\ldots,m-1}, \{v-2^{m}, v-2^{m}-1\} \Big\},
\]
\[
E(\mathcal{P}^{l}_v)=\Big\{ \{v+2^{i},v+2^{i+1}\}_{i=0,1,\ldots, l-1} \Big\}
\]
that are the edges of two paths,
\[
E(\mathcal{G}_{v})=\Big\{ \{v+2^{i},v+2^{i+1}\}_{i=0,1,\ldots, m-1}\Big\} \cup \Big\{\{v+2^{m},v+2^{m}-1\}\Big\},
\]
\[
E(\mathcal{L}_{v})=\Big\{ \{v-2^{i},v-2^{i+1}\}_{i=0,1,\ldots, m-1}\Big\} \cup \Big\{\{v-2^{m},v-2^{m}+1\}\Big\},
\]
that are the edges of two cycles with an extra edge and
\[
E(\mathcal{B}^l_{v})=\Big\{ \{v-2^{i},v+2^{i}\}_{i=l,l+1,\ldots,m} \Big\}
\]
that are disjoint edges connecting $E(\mathcal{L}_{v})$ and $E(\mathcal{G}_{v})$.
We will prove that $\Delta$ is vertex decomposable by using Lemma \ref{weldel}, that is we want to find a sequence of vertices $v_1,v_2,\ldots ,v_{n-3}$ satisfying (ii).$(1)$ and (ii).$(2)$.
We claim that such a sequence is 
\[
1,2,\ldots,\widehat{2}^m,2^{m}+1,\ldots,\widehat{2}^{m+1},\ldots,  n-1.
\]
Let us consider the vertices $v$ in $1, 2, \ldots, 2^{m}-1$.\\
For $v=1$ and $\Delta_{0}=\Delta$, $\link_\Delta (1)$ is vertex decomposable. In fact, for any $v \in V(G)$, $\mathcal{F}(\link_\Delta (v))$ is, by abuse of notation, 
\[
E(\mathcal{L}_{v}) \cup E(\mathcal{G}_{v}) \cup E(\mathcal{B}_{v}^{0}) \cup \Big\{\{ v-1,v+2^{m}-1 \}, \{ v+1,v-2^{m}+1 \} \Big\}
\]
that is $1$-dimensional and connected (see Figure \ref{v1}).\\
We describe the first steps $v=2,3,4$ before giving the general set \eqref{E1} for $\FF(\link_{\Delta_{v-1}}(v))$ with $v$ in the interval $[2,2^{m}-1]$. \\
For $v=2$, we have that the vertex $1=v-1$ is not in $\link_{\Delta_{1}}(v)$, hence $\FF(\link_{\Delta_{1}}(v))$ is equal to
\[
E(\mathcal{H}_{v}^{1})\cup E(\mathcal{B}^1_{v}) \cup \Big\{\{v-2^{m}+1,v+1 \}\Big\} \cup E(\mathcal{G}_{v}),
\]
that is 1-dimensional and connected (see Figure \ref{v2}).  From now on, we omit the last observation that will be clear by the descriptions of the links.\\
For $v= 3$, we have that $2=v-1, \ \ 1=v-2 \notin V( \Delta_2)$, hence the edges $\{v-2,v+4\}$ and $\{v-2,v+2\} $ are not in $\Delta_2$ (see Figure \ref{v3}), and $\FF(\link_{\Delta_{2}}(v))$ is
\[
E(\mathcal{H}_{v}^{2})\cup E(\mathcal{B}^2_{v})  \cup  \Big\{\{v-2^{m}+1,v+1 \} \Big\} \cup E(\mathcal{G}_{v}).
\]
For $v=4$, the same facts of the case $v=3$ hold. Hence, $\link_{\Delta_{3}}(v)$ is isomorphic to $\link_{\Delta_{2}}(3)$.
To get the general set we observe that for $2\leq v \leq 2^{m}-1$, we have two cases: if $v-1 \in \bar{S}$ we loose the edge $\{v-2^{l},v-2^{l+1} \}$ of $E(\mathcal{H}^l_{v})$ and the edge $\{v-2^{l},v+2^{l} \}$ of $E(\mathcal{B}^l_{v})$ from the edges of $\link_{\Delta_{v-2}}(v-1)$, as in the cases $v=2,3$; otherwise $\link_{\Delta_{v-1}}(v)$ is isomorphic to $\link_{\Delta_{v-2}}(v-1)$, as in the case $v=4$.
To get the set \eqref{E1}, we pose
\[
j(v)=\min\{l\in \mathbb{N}: 2^{l-1} \leq v-1 < 2^{l}\}
\]
for $2\leq v \leq 2^{m}-1$. Moreover we have $1\leq v-1 < 2^m$, so that $j(v)\leq m$, and $\FF(\link_{\Delta_{v-1}}(v))$ is equal to
\begin{equation}\label{E1}
E(\mathcal{H}_{v}^{j(v)})\cup E(B_{v}^{j(v)}) \cup  \Big\{\{v-2^{m}+1,v+1 \} \Big\} \cup E(\mathcal{G}_{v})
\end{equation}
that is connected because $\mathcal{G} $ and $\mathcal{H}$ are, and they are joined by $ \{v-2^{m}+1,v+1 \}$.
Now, we consider  the vertices $v$ in $2^m+1, 2^{m}+2, \ldots,  2^{m+1}-1$. \\
For $v=2^m +1$, since $2^m$ has not been removed, $2^m=v-1 \in V(\link_{\Delta_{v-2}}(v))$ (see Figure \ref{v9}), and we have
\[
\FF(\link_{\Delta_{v-2}}(v))=\Big\{\{v-1,v+1\}\Big\} \cup \Big\{\{v-1,v+2^{m}-1\}\Big\}  \cup E(\mathcal{G}_{v}).
\]
We exploit the steps $v=2^{m}+2,2^{m}+3$ before giving the general set \eqref{E2} for $2^m+2 \leq v \leq 2^{m+1}-2$.\\
For $v=2^m +2$, since $2^m$ has not been removed, $2^{m}=v-2 \in V(\link_{\Delta_{v-2}}(v))$ (see Figure \ref{v10}) and
\[
\FF(\link_{\Delta_{v-2}}(v))=\Big\{\{v-2,v+2\}\Big\}\ \cup E( \mathcal{G}_{v}).
\]
For $v=2^m +3$, $v$ is not adjacent to $2^m$ (since $3\notin \bar{S}$) and since we removed the vertex $w$ for $1\leq w\leq 2^{m}-1$,
\[
\FF(\link_{\Delta_{v-2}}(v))= E( \mathcal{G}_{v})
\]
(see Figure \ref{v11}).
In general, for $2^m+2 \leq v \leq 2^{m+1}-2$, since the only vertex in $\{1,\ldots , 2^m\}$ that we have not removed is $2^m$, when $v =2^m+2^j$, then $v-2^{j} \in V(\link_{\Delta_{v-2}} (v))$, that is
\begin{equation}\label{E2}
\FF(\link_{\Delta_{v-2}}(v))=\begin{cases} \Big\{\{v-2^j,v+2^j\}\Big\} \cup E(\mathcal{G}_{v}) &\mbox{if } v=2^m+2^j  \\  E(\mathcal{G}_{v}) &\mbox{otherwise.}  \end{cases} 
\end{equation}
For $v=2^{m+1}-1$, since $v$ is adjacent to $2^m$, we have
\[
\FF(\link_{\Delta_{v-2}}(v))=\Big\{\{v-2^m+1,v+1\}\Big\} \cup E(\mathcal{G}_{v})
\]
(see Figure \ref{v15}).
To complete the decomposition, we need to remove the vertices $v$ in $2^{m+1}+1,\ldots, n-1$. \\
For $v=2^{m+1}+1$, since $2^{m+1}$ has not been removed, $2^{m+1}=v-1 \in V(\link_{\Delta_{v-3}}(v))$. On the other hand, $1=v+2^m$ has been removed (Figure \ref{v17}), hence
\[
\FF(\link_{\Delta_{v-3}}(v))=\Big\{\{v+2^{m}-1,v-1\}\Big\} \cup \Big\{\{v-1,v+1\}\Big\} \cup E(\mathcal{P}_{v}^{m-1}).
\]
For $2^{m+1}+2\leq v \leq n-2=2^{m+1}+2^{m}-2$, we set
\[
k(v)=\min \{l\in \mathbb{N} : 2^{l} \leq n-v < 2^{l+1} \}.
\]
We observe that the path $\mathcal{P}_{v}^{k(v)}$ is contained in $\link_{\Delta_{v-3}}(v)$ (see Figure \ref{v18} and Figure \ref{v19}). Moreover, if $v =2^{m+1}+2^{j}$, then $v$ is adjacent to $2^{m+1}$, and $\link_{\Delta_{v-3}}(v)$ contains the edge $\{v-2^{j},v+2^j\}$ (see Figure \ref{v18}). Hence, for $2^{m+1}+2\leq v \leq n-2$ we have
\[
\FF(\link_{\Delta_{v-3}}(v))=\begin{cases} \Big\{\{v-2^j,v+2^j\}\Big\} \cup E(\mathcal{P}_{v}^{k(v)})&\mbox{if } v=2^{m+1}+2^j  \\  E(\mathcal{P}_{v}^{k(v)}) &\mbox{otherwise.} \end{cases}
\]
The existence of $\{v-2^j,v+2^j\}$ is guaranteed by the inequality $2^{j}\leq 2^{k(v)}$. In fact, $n-v=2^{m+1}+2^m-2^{m+1}-2^j=2^{m}-2^j$. Since $j\leq m-1$, then
\[2^{m-1} \leq 2^{m}-2^{j}\leq 2^m.\]
So $k(v)=m-1$ for $v=2^{m+1}+2^j$. \\ The last vertex that we remove is $v=n-1$. We have not removed yet $2^{m+1}=v-2^{m}+1$ and $0=v+1$. Therefore its link is formed only by the edge $\{v-2^{m}+1,v+1\}$ (see Figure \ref{v23}). The only vertices that we have not removed are $\{0,2^{m},2^{m+1}\}$, that is a simplex on $3$ vertices. The assertion follows.
\end{proof}

Now we give an example of circulant belonging to the class above.

\begin{Example}\label{exam}
The first circulant of the class is $G=C_{24}(\ol{1,2,4,7,8})$ and let $\Delta=\Delta (G)$. 
We want to prove that the sequence
\[
\mathcal{M}=1,2,\ldots,\widehat{8},9,\ldots,\widehat{16},17,\ldots,23].
\]
satisfies the conditions $(1)$ and $(2)$ of Lemma \ref{weldel}.\\
By using the notation of the proof of Theorem \ref{2vd}, we have 
\[
E(\mathcal{H}_{v}^{1})=\Big\{ \{v-2,v-4\}, \{v-4,v-8\},\{v-8, v-7\} \Big\},
\]
\[
E(\mathcal{G}_{v})=\Big\{ \{v+1,v+2\},\{v+2,v+4\},\{v+4,v+8\},\{v+8,v+1\},\{v+8,v+7\}, \Big\},
\]
\[
E(\mathcal{P}^{2}_v)=\Big\{\{v+1,v+2\},\{v+2,v+4\} \Big\},
\]
\[
E(\mathcal{P}^{1}_v)=\Big\{\{v+1,v+2\}  \Big\},
\]
\[
E(\mathcal{B}^1_{v})=\Big\{\{v-2,v+2\},\{v-4,v+4\},\{v-8,v+8\}  \Big\}
\]
For $v=1$, $\link_{\Delta}(v)$ is represented in Figure \ref{v1} and it is $1$-dimensional and connected. From now on, we omit last observation that will be clear by the figures.
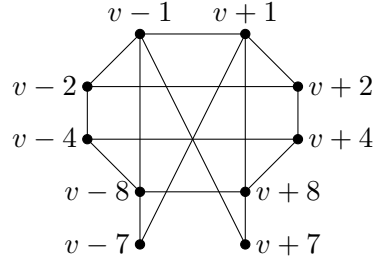
\begin{figure}[h]
\begin{center}
\begin{tikzpicture}[scale=0.7]
\draw[black!100] (-1,1) -- (1,1);
\draw[black!100]  (-2,-1) --(2,-1);
\draw[black!100] (1,1) -- (2,0);
\draw[black!100] (2,-1) -- (2,0);
\draw[black!100]  (-2,0) -- (2,0);
\draw[black!100]  (-2,0)--(-2,-1);
\draw[black!100]  (-2,0)--(-1,1);
\draw[black!100]  (-1,-2)--(-2,-1);
\draw[black!100] (2,-1) -- (1,-2);
\draw[black!100]  (-1,-2) -- (1,-2);
\draw[black!100]  (-1,-3) -- (-1,-2);
\draw[black!100] (1,-3) -- (1,-2);
\draw (-1,-2) -- (-1,1);
\draw[black!100] (1,-2) -- (1,1);
\draw (1,-3) -- (-1,1);
\draw (-1,-3) -- (1,1);
\filldraw (-1,-3) circle (2.5 pt) node [anchor=east] {$v-7$};
\filldraw (1,-3) circle (2.5 pt) node [anchor=west] {$v+7$};
\filldraw (-1,-2) circle (2.5 pt) node [anchor=east] {$v-8$};
\filldraw (1,-2) circle (2.5 pt) node [anchor=west] {$v+8$};
\filldraw (-2,0) circle (2.5 pt) node [anchor=east] {$v-2$};
\filldraw (2,0) circle (2.5 pt) node [anchor=west] {$v+2$};
\filldraw (1,1) circle (2.5 pt) node [anchor=south] {$v+1$};
\filldraw (-1,1) circle (2.5 pt) node [anchor=south] {$v-1$};
\filldraw (2,-1) circle (2.5 pt) node [anchor=west] {$v+4$};
\filldraw (-2,-1) circle (2.5 pt) node [anchor=east] {$v-4$};
\end{tikzpicture}\caption{$\link_\Delta (v)$ for any $v \in V(G)$.}\label{v1}
\end{center}
\end{figure}
\\Then we define $\Delta_1=\del_{\Delta}(1)$ and we study $\link_{\Delta_{v-1}}(v)$ when $v=2$. Since $v-1=1$ and we have removed the vertex $1$, $v-1$ does not appear in $\link_{\Delta_{v-1}}(v)$ (Figure \ref{v2}).
We point out that the latter is formed by $\mathcal{H}_{v}^{1}$, $\mathcal{B}_{v}^1$, $\mathcal{G}_{v}$ and the edge $\{v-7,v+1\}$.
\\Then we set $\Delta_2=\del_{\Delta}(2)$ and we look at $\link_{\Delta_{v-1}}(v)$ when $v=3$. Since we removed $1=v-2$ and $2=v-1$, then $v-1,v-2 \notin \link_{\Delta_{v-1}}(v)$ (Figure \ref{v3}).
\begin{figure}[h]
\begin{center}
\begin{tikzpicture}[scale=0.7]
\draw (-2,-1) --(2,-1);
\draw (1,1) -- (2,0);
\draw (2,-1) -- (2,0);
\draw (-2,0) -- (2,0);
\draw (-2,0)--(-2,-1);
\draw (-1,-2)--(-2,-1);
\draw (2,-1) -- (1,-2);
\draw (-1,-2) -- (1,-2);
\draw (-1,-3) -- (-1,-2);
\draw (1,-3) -- (1,-2);
\draw (1,-2) -- (1,1);
\draw (-1,-3) -- (1,1);
\filldraw (-1,-3) circle (2.5 pt) node [anchor=east] {$v-7$};
\filldraw (1,-3) circle (2.5 pt) node [anchor=west] {$v+7$};
\filldraw (-1,-2) circle (2.5 pt) node [anchor=east] {$v-8$};
\filldraw (1,-2) circle (2.5 pt) node [anchor=west] {$v+8$};
\filldraw (-2,0) circle (2.5 pt) node [anchor=east] {$v-2$};
\filldraw (2,0) circle (2.5 pt) node [anchor=west] {$v+2$};
\filldraw (1,1) circle (2.5 pt) node [anchor=south] {$v+1$};
\filldraw (2,-1) circle (2.5 pt) node [anchor=west] {$v+4$};
\filldraw (-2,-1) circle (2.5 pt) node [anchor=east] {$v-4$};
\end{tikzpicture}\caption{$\link_{\Delta_{v-1}}(v)$ for $v=2$.}\label{v2}
\end{center}
\end{figure}
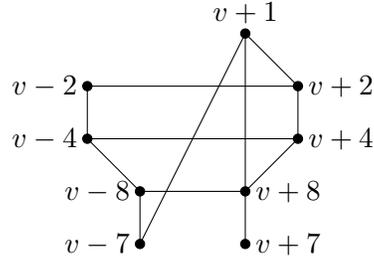
\begin{figure}[h]
\begin{center}
\begin{tikzpicture}[scale=0.7]
\draw (-2,-1) --(2,-1);
\draw (1,1) -- (2,0);
\draw (2,-1) -- (2,0);
\draw (-1,-2)--(-2,-1);
\draw (2,-1) -- (1,-2);
\draw (-1,-2) -- (1,-2);
\draw (-1,-3) -- (-1,-2);
\draw (1,-3) -- (1,-2);
\draw (1,-2) -- (1,1);
\draw (-1,-3) -- (1,1);
\filldraw (-1,-3) circle (2.5 pt) node [anchor=east] {$v-7$};
\filldraw (1,-3) circle (2.5 pt) node [anchor=west] {$v+7$};
\filldraw (-1,-2) circle (2.5 pt) node [anchor=east] {$v-8$};
\filldraw (1,-2) circle (2.5 pt) node [anchor=west] {$v+8$};
\filldraw (2,0) circle (2.5 pt) node [anchor=west] {$v+2$};
\filldraw (1,1) circle (2.5 pt) node [anchor=south] {$v+1$};
\filldraw (2,-1) circle (2.5 pt) node [anchor=west] {$v+4$};
\filldraw (-2,-1) circle (2.5 pt) node [anchor=east] {$v-4$};
\end{tikzpicture}\caption{$\link_{\Delta_{v-1}}(v)$ for $v=3,4$.}\label{v3}
\end{center}
\end{figure}
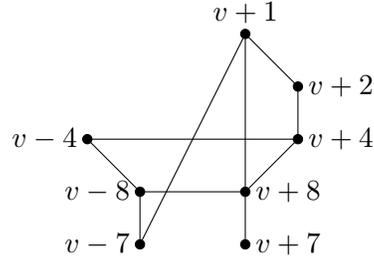
\\For $v=4$, the same facts of the case $v=3$ hold. Hence $\link_{\Delta_{v-1}}(v)$ is isomorphic to $\link_{\Delta_{v-2}}(v-1)$ (Figure \ref{v3}). \\For $v=5$, since we have removed $1=v-4, \ 3=v-2$ and $4=v-1$, then they do not appear in $\link_{\Delta_{v-1}}(v)$ (Figure \ref{v5}).
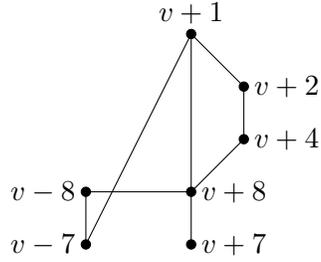
\begin{figure}[h]
\begin{center}
\begin{tikzpicture}[scale=0.7]
\draw (1,1) -- (2,0);
\draw (2,-1) -- (2,0);
\draw (2,-1) -- (1,-2);
\draw (-1,-2) -- (1,-2);
\draw (-1,-3) -- (-1,-2);
\draw (1,-3) -- (1,-2);
\draw (1,-2) -- (1,1);
\draw (-1,-3) -- (1,1);
\filldraw (-1,-3) circle (2.5 pt) node [anchor=east] {$v-7$};
\filldraw (1,-3) circle (2.5 pt) node [anchor=west] {$v+7$};
\filldraw (-1,-2) circle (2.5 pt) node [anchor=east] {$v-8$};
\filldraw (1,-2) circle (2.5 pt) node [anchor=west] {$v+8$};
\filldraw (2,0) circle (2.5 pt) node [anchor=west] {$v+2$};
\filldraw (1,1) circle (2.5 pt) node [anchor=south] {$v+1$};
\filldraw (2,-1) circle (2.5 pt) node [anchor=west] {$v+4$};
\end{tikzpicture}\caption{$\link_{\Delta_{v-1}}(v)$ for $v=5,6,7$.}\label{v5}
\end{center}
\end{figure}
\\
For $v=6$ and $7$, the same facts of the case $v=5$ hold and their links are isomorphic to the one in Figure \ref{v5}. \\
Now we jump from $v=7$ to $v=9$, without removing the vertex $8$. It implies that we have $v-1=8$ is in $\link_{\Delta_{v-2}}(v)$ (Figure \ref{v9}).
\begin{figure}[h]
\begin{center}
\begin{tikzpicture}[scale=0.7]
\draw (1,1)-- (-1,1);
\draw (1,1) -- (2,0);
\draw (2,-1) -- (2,0);
\draw (2,-1) -- (1,-2);
\draw (1,-3) -- (1,-2);
\draw (1,-2) -- (1,1);
\draw (1,-3) -- (-1,1);
\filldraw (1,-3) circle (2.5 pt) node [anchor=west] {$v+7$};
\filldraw (1,-2) circle (2.5 pt) node [anchor=west] {$v+8$};
\filldraw (2,0) circle (2.5 pt) node [anchor=west] {$v+2$};
\filldraw (1,1) circle (2.5 pt) node [anchor=south] {$v+1$};
\filldraw (-1,1) circle (2.5 pt) node [anchor=south] {$v-1$};
\filldraw (2,-1) circle (2.5 pt) node [anchor=west] {$v+4$};
\end{tikzpicture}\caption{$\link_{\Delta_{v-2}}(v)$ for $v=9$.}\label{v9}
\end{center}
\end{figure}
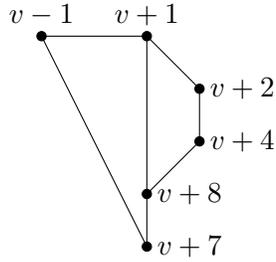
\\ From $v=10$ to $v=14$, $\link_{\Delta_{v-2}}(v)$ is formed by $\mathcal{G}_v$ (Figure \ref{v11}), and the edge connecting $8$ to $\mathcal{G}_v$, when $v=8+2^j$, namely $\{8=v-2^{j},v+2^{j}\}$ (Figure \ref{v10} and Figure \ref{v12}).
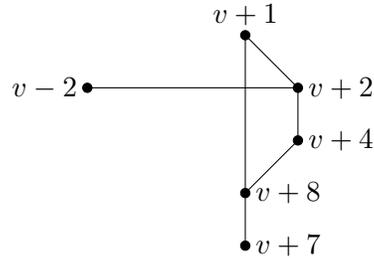
\begin{figure}[h]
\begin{center}
\begin{tikzpicture}[scale=0.7]
\draw (-2,0) -- (2,0);
\draw (1,1) -- (2,0);
\draw (2,-1) -- (2,0);
\draw (2,-1) -- (1,-2);
\draw (1,-3) -- (1,-2);
\draw (1,-2) -- (1,1);
\filldraw (1,-3) circle (2.5 pt) node [anchor=west] {$v+7$};
\filldraw (1,-2) circle (2.5 pt) node [anchor=west] {$v+8$};
\filldraw (2,0) circle (2.5 pt) node [anchor=west] {$v+2$};
\filldraw (1,1) circle (2.5 pt) node [anchor=south] {$v+1$};
\filldraw (2,-1) circle (2.5 pt) node [anchor=west] {$v+4$};
\filldraw (-2,0) circle (2.5 pt) node [anchor=east] {$v-2$};
\end{tikzpicture}\caption{$\link_{\Delta_{v-2}}(v)$ for $v=10$.}\label{v10}
\end{center}
\end{figure} \\
\begin{figure}[h]
\begin{center}
\begin{tikzpicture}[scale=0.7]
\draw (0,1) -- (1,0);
\draw (1,-1) -- (1,0);
\draw (1,-1) -- (0,-2);
\draw (0,-3) -- (0,-2);
\draw (0,-2) -- (0,1);
\filldraw (0,-3) circle (2.5 pt) node [anchor=west] {$v+7$};
\filldraw (0,-2) circle (2.5 pt) node [anchor=west] {$v+8$};
\filldraw (1,0) circle (2.5 pt) node [anchor=west] {$v+2$};
\filldraw (0,1) circle (2.5 pt) node [anchor=south] {$v+1$};
\filldraw (1,-1) circle (2.5 pt) node [anchor=west] {$v+4$};
\end{tikzpicture}\caption{$\link_{\Delta_{v-2}}(v)$ for $v=11,13,14$.}\label{v11}
\end{center}
\end{figure}
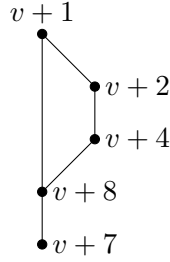
\begin{figure}[h]
\begin{center}
\begin{tikzpicture}[scale=0.7]
\draw (-2,-1) -- (2,-1);
\draw (1,1) -- (2,0);
\draw (2,-1) -- (2,0);
\draw (2,-1) -- (1,-2);
\draw (1,-3) -- (1,-2);
\draw (1,-2) -- (1,1);
\filldraw (1,-3) circle (2.5 pt) node [anchor=west] {$v+7$};
\filldraw (1,-2) circle (2.5 pt) node [anchor=west] {$v+8$};
\filldraw (2,0) circle (2.5 pt) node [anchor=west] {$v+2$};
\filldraw (1,1) circle (2.5 pt) node [anchor=south] {$v+1$};
\filldraw (2,-1) circle (2.5 pt) node [anchor=west] {$v+4$};
\filldraw (-2,-1) circle (2.5 pt) node [anchor=east] {$v-4$};
\end{tikzpicture}\caption{$\link_{\Delta_{v-2}}(v)$ for $v=12$.}\label{v12}
\end{center}
\end{figure}
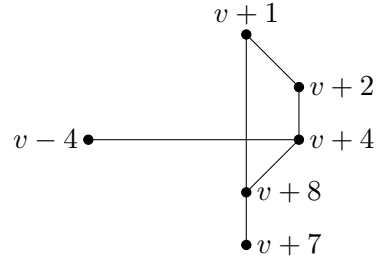
\ \\
For $v=15$, $8=v-7 $ appears in $\link_{\Delta_{v-2}}(v)$ and once again it is connected to $\mathcal{G}_v$  (Figure \ref{v15}).
\begin{figure}[h]
\begin{center}
\begin{tikzpicture}[scale=0.7]
\draw (-1,-3) -- (1,1);
\draw (1,1) -- (2,0);
\draw (2,-1) -- (2,0);
\draw (2,-1) -- (1,-2);
\draw (1,-3) -- (1,-2);
\draw (1,-2) -- (1,1);
\filldraw (1,-3) circle (2.5 pt) node [anchor=west] {$v+7$};
\filldraw (1,-2) circle (2.5 pt) node [anchor=west] {$v+8$};
\filldraw (2,0) circle (2.5 pt) node [anchor=west] {$v+2$};
\filldraw (1,1) circle (2.5 pt) node [anchor=south] {$v+1$};
\filldraw (2,-1) circle (2.5 pt) node [anchor=west] {$v+4$};
\filldraw (-1,-3) circle (2.5 pt) node [anchor=east] {$v-7$};
\end{tikzpicture}\caption{$\link_{\Delta_{v-2}}(v)$ for $v=15$.}\label{v15}
\end{center}
\end{figure}
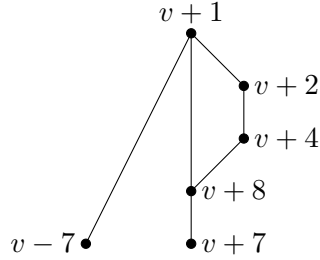
\\
We jump from $v=15$ to $v=17$. Since we have not removed $16=v-1$, then $v-1$ is contained in $\link_{\Delta_{v-3}}(v)$ (analogously to the case $v=9$).
On the other hand, since we have removed $1=v+8$, it does not appear in $\link_{\Delta_{v-3}}(v)$ (Figure \ref{v17}), that is the path $\mathcal{P}_{v}^{2}$ plus the edges $\{v+1,v-1\}$ and $\{v-1,v+7\}$.
\begin{figure}
\begin{center}
\begin{tikzpicture}[scale=0.7]
\draw (1,1)-- (-1,1);
\draw (1,1) -- (2,0);
\draw (2,-1) -- (2,0);
\draw (1,-3) -- (-1,1);
\filldraw (1,-3) circle (2.5 pt) node [anchor=west] {$v+7$};
\filldraw (2,0) circle (2.5 pt) node [anchor=west] {$v+2$};
\filldraw (1,1) circle (2.5 pt) node [anchor=south] {$v+1$};
\filldraw (-1,1) circle (2.5 pt) node [anchor=south] {$v-1$};
\filldraw (2,-1) circle (2.5 pt) node [anchor=west] {$v+4$};
\end{tikzpicture}\caption{$\link_{\Delta_{v-3}}(v)$ for $v=17$.}\label{v17}
\end{center}
\end{figure}
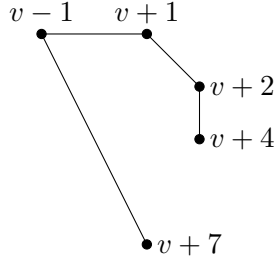
\\
From $v=18$ to $v=20$, $\link_{\Delta_{v-3}}(v)$ is formed by $\mathcal{P}^2_v$ (Figure \ref{v19}), and the edge connecting $16$ to $\mathcal{P}^2_v$, when $v=16+2^j$, namely $\{16=v-2^{j},v+2^{j}\}$ (e.g. Figure \ref{v18}).
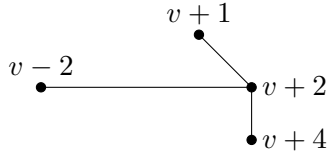
\begin{figure}[h]
\begin{center}
\begin{tikzpicture}[scale=0.7]
\draw (2,0)--(-2,0);
\draw (1,1) -- (2,0);
\draw (2,-1) -- (2,0);
\filldraw (2,0) circle (2.5 pt) node [anchor=west] {$v+2$};
\filldraw (1,1) circle (2.5 pt) node [anchor=south] {$v+1$};
\filldraw (-2,0) circle (2.5 pt) node [anchor=south] {$v-2$};
\filldraw (2,-1) circle (2.5 pt) node [anchor=west] {$v+4$};
\end{tikzpicture}\caption{$\link_{\Delta_{v-3}}(v)$ for $v=18$.}\label{v18}
\end{center}
\end{figure}
\begin{figure}[h]
\begin{center}
\begin{tikzpicture}[scale=0.7]
\draw (1,1) -- (2,0);
\draw (2,-1) -- (2,0);
\filldraw (2,0) circle (2.5 pt) node [anchor=west] {$v+2$};
\filldraw (1,1) circle (2.5 pt) node [anchor=south] {$v+1$};
\filldraw (2,-1) circle (2.5 pt) node [anchor=west] {$v+4$};
\end{tikzpicture}\caption{$\link_{\Delta_{v-3}}(v)$ for $v=19$.}\label{v19}
\end{center}
\end{figure}
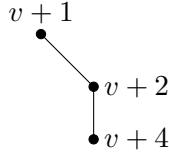
\\For $v=21,22$, they are not adjacent to $16$ and we have removed $1,2=v+4$. That is $\link_{\Delta_{v-3}}(v)$ is only $\mathcal{P}_{v}^{1}$.
\\For $v=23$, the only vertices not yet removed are $0=v+1$ and $16=v-7$ (Figure \ref{v23}).
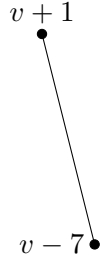
\begin{figure}[h]
\begin{center}
\begin{tikzpicture}[scale=0.7]
\draw (1,-3) -- (0,1);
\filldraw (0,1) circle (2.5 pt) node [anchor=south] {$v+1$};
\filldraw (1,-3) circle (2.5 pt) node [anchor=east] {$v-7$};
\end{tikzpicture}\caption{$\link_{\Delta_{v-3}}(v)$ for $v=23$.}\label{v23}
\end{center}
\end{figure}
Hence we have proved $(\mathrm{ii}).(1)$ of Lemma \ref{weldel}.
Then we are left with the only triangle $\{0,8,16\}$, that is a simplex on $3$ vertices, so that $(\mathrm{ii}).(2)$ of Lemma \ref{weldel} is satisfied. Hence $\Delta$ is vertex decomposable.
\end{Example}

\section{Level algebras and Gorenstein}\label{sec:levG}
In this section we prove that any 2-dimensional vertex decomposable independence complex of circulants has a level Stanley-Reisner ring.

\begin{Definition}
We say that a graph $G$ is \emph{$l$-connected} if for every subset
$S \subseteq V(G)$ of cardinality $|S| < l$, then $G$ restricted to the set of vertices $V \setminus S$ is connected.
We simply call \emph{connected graph} a $1$-connected graph, \emph{biconnected graph} a 2-connected graph, \emph{triconnected graph} a 3-connected graph.
\end{Definition}

\begin{Lemma}\label{tric}
Let $G=C_n(a,b)$ be a connected circulant graph. Then $G$ is triconnected.
\end{Lemma}
\begin{proof}
The graph $G$ is connected if and only if $\gcd(n,a,b)=1$.
 We have to prove that after we remove any two vertices the graph remains connected. We take out the vertex $0$. Since any connected circulant is biconnected, the remaining graph is connected. Let $G^* = G\setminus \{0\}$.
We have two cases:
\begin{itemize}
\item[(T1)] One of the elements $a,b$ is coprime with $n$.
\item[(T2)] Neither $a$ nor $b$ is coprime with $n$.
\end{itemize}
(T1) We assume $a$ coprime with $n$, then	
\[
V(G)=\{a,2a,\ldots,(n-1)a \}.
\] 
Let $b=sa$.
By removing the vertex $a$ (respectively the vertex $(n-1)a$), the remaining graph is connected through the path $\{2a,3a,\ldots,(n-1)a\}$ (respectively $\{a,2a,\ldots,(n-2)a\}$). So we need to consider the removal of $ia$ with $2 \leq i \leq n-2$. We end up with the two paths on vertices
\[
 A=\{a,\ldots, (i-1)a\}, \ \ B=\{(i+1)a,\ldots, (n-1)a\}.
\]
We prove that the two paths above are connected each other by some edges. We have two cases: $i\leq s$, $i > s$.
If $i\leq s$, then $i+1 \leq s+1 \in B$ and $\{a,a+b\}=\{a,(s+1)a\}\in E(G)$.
If $i>s$, since $1 <s  \leq i-1$, then  $(i+1-s)a \in A $. Hence $\{(i+1-s)a,(i+1)a\} \in E(G)$. The assertion follows

(T2) We assume $d=\gcd(b,n)$ and since $\gcd(a,b,n)=1$, then $a$ is coprime with $d$. Let $n=ld$. It follows that the vertex set $V(G^*)$ can be partitioned in 
\[
\begin{matrix}
V_{0}&= & &\{ d, &\ldots &(l-1)d \}, \\
V_1&=&\{a,&a+d, &\ldots &a+(l-1)d \}, \\
&&&& \vdots\\
V_{d-1}&=&\{(d-1)a, &(d-1)a+d, &\ldots &(d-1)a+(l-1)d \}.\\
\end{matrix}
\]

We observe that the sets $V_{i}$ and $V_{i+1}$ are connected each other since $a \in S$. Moreover each $V_i$ is connected, since $d \ | \ b \in S$, and if $r \neq 0$, then $V_{r}$ is a cycle and it is biconnected. It implies that after removing a vertex from $V_{r}$ with $r \neq 0$, the graph remains connected. So we assume $r=0$, and we remove the vertex $kd$ for some $k$. Hence we have to prove that the two sets 
\[
V_{0}'=\{d,\ldots, (k-1)d\},\ \  V_{0}''=\{(k+1)d,\ldots ,(l-1)d\}
\]
are connected each other. Take $x \in V_{0}'$ and $y \in V_{0}''$. Then $a+x,a+y \in V_{1} $. Since $V_1$ is connected, the assertion follows.
\end{proof}
We present an interesting observation on the reduced Euler characteristic of $2$-dimensional complexes for circulants.
\begin{Lemma}\label{EC}
Let $n \geq 6$ and $G=C_{n}(S)$ be circulant graph $\dim \Delta=2$.
Then\[
\widetilde{\chi}(\Delta) \neq 0.
\]
\end{Lemma}
\begin{proof}
If $n=6$ it is an easy task.
We consider the case $n>6$. The $\mathbf{f}$-vector of $\Delta$ is given by $(f_{0},f_{1},f_{2})$ and
\[
\widetilde{\chi}(\Delta)=-1+f_{0}-f_{1}+f_{2}=-1+n-f_{1}+f_{2}.
\]
Let $d=\gcd (f_{2},f_{1},n)$. By Lemma 2.2 of \cite{Ri} it follows that 
\[
f_{1}=\frac{nf_{1,0}}{2}, \ \ \  f_{2}=\frac{nf_{2,0}}{3}
\]
Hence in any case $d \in \{\frac{n}{6}, \frac{n}{3},\frac{n}{2},n   \}$. Since $n>6$, then $d>1$. Therefore it follows that 
\[
\widetilde{\chi}(\Delta)\equiv -1 \ \ \mod \ d
\]
and $\widetilde{\chi}(\Delta) \neq 0$.
\end{proof}

\begin{Lemma}\label{2S}
Let $G=C_{n}(\bar{S})$ be a circulant graph such that $\Delta$ is 2-dimensional and vertex decomposable. Then $|S|\geq 2$.
\end{Lemma}
\begin{proof}
By contraposition, let us assume that $G=C_{n}(a)$, for $a\in \mathbb{Z}_n$. Since $\dim \Delta=2$, then $n=3a$. It implies that $\Delta$ is disconnected, that implies $\Delta$ is not vertex decomposable.
\end{proof}

\begin{Proposition}\label{reg}
Let $G=C_{n}(\bar{S})$ be a circulant graph such that $\Delta$ is 2-dimensional and Cohen-Macaulay. Then 
\[
\reg R/I(G)=3.
\]
\end{Proposition}
\begin{proof}
Since $R/I(G)$ is Cohen-Macaulay, then by Lemma \ref{EC}, \cite[Remark 1.2]{RR} and \cite[Corollary 4.8]{Ei} we get 
\[
\reg R/I(G)=\depth R/I(G)=\dim R/I(G)=3.
\]
\end{proof}

\begin{Theorem}\label{typ}
Let $G=C_{n}(\bar{S})$ be a circulant graph such that $\Delta$ is 2-dimensional and vertex decomposable. Then $R/I(G)$ is a level algebra.
\end{Theorem}
\begin{proof}
From Remark \ref{EC}, $\widetilde{\chi}(\Delta)$ is always non-zero. Hence according to Remark \ref{leval}, we have to prove that the Cohen-Macaulay type of $R/I(G)$ coincides with $\widetilde{\chi}(\Delta)$.
Namely, we have to compute the last total Betti number of the minimal free resolution of $R/I(G)$.
Since $\Delta$ is also Cohen-Macaulay, then
\[
\depth R/I(G)=\dim R/I(G)=3.
\]
From Auslander-Buchsbaum formula, we have that 
\[
\pd R/I(G)=\dim R - \depth R/I(G)=n-3.
\]
From Proposition \ref{reg} we have $\reg R/I(G)=3$.
So we have to look at the Betti numbers $\beta_{n-3,j}$ for $j \in \{1,2,3\}$.
According to Hochster's Formula (Theorem \ref{hoch}),
\[
\beta_{i,\sigma}=\dim_{K}  \widetilde{H}_{|\sigma|-i-1}(\Delta_{|\sigma};K).
\]
First of all, we assume $|\sigma|=n-2$ and so 
\[
\beta_{n-3,n-2}=\dim_{K}  \widetilde{H}_{0}(\Delta_{|\sigma};K)=0
\]
because from Lemma \ref{2S} and Lemma \ref{tric}, the graph $\bar{G}$ is triconnected.
Now we assume $|\sigma|=n-1$, and so
\[
\beta_{n-3,n-1}=\dim_{K}  \widetilde{H}_{1}(\Delta_{|\sigma};K).
\]
Since $\Delta_{|\sigma}$ is the simplicial complex defined on $V(G)=\{0,1,2,\ldots,\widehat{x},\ldots,n-1\}$, it holds
\[
\Delta_{|\sigma} \simeq \del_{\Delta}(x)
\]
that is vertex decomposable because $\Delta$ is, and hence Cohen-Macaulay. Therefore $\beta_{n-3,n-1}=0$.
Finally, we assume $|\sigma|=n$, and hence 
\[
\beta_{n-3,n}=\dim_{K}  \widetilde{H}_{2}(\Delta;K)=\widetilde{\chi}(\Delta),
\]
and the assertion follows.
\end{proof}

\begin{Corollary}\label{type}
Let $G=C_{n}(\ol{1,2,4,\ldots,2^m,2^{m}-1})$ with $m\geq 3$ and $n=3\cdot 2^{m}$. Then $R/I(G)$ is
a level algebra.
\end{Corollary}

It is of interest to know whether a level algebra is also a Gorenstein algebra.
In general, we have the following
\begin{Theorem}\label{Gor}
Let $G$ be a non-empty circulant graph with $\dim R/I(G)=3$. The following are equivalent:
\begin{enumerate}
\item $R/I(G)$ is Gorenstein;
\item $G=C_{6}(3)$.
\end{enumerate}
\end{Theorem}
\begin{proof}
(2)$\Rightarrow$(1). It is easy to verify that $\Delta (C_{6}(3))$ is vertex decomposable. Then, according to Theorem \ref{typ} and Remark \ref{leval}, we have to compute $\widetilde{\chi}(\Delta)$. The $f$-vector of $\Delta$ is $(1,6,12,8),$
so $\widetilde{\chi}(\Delta)=-1+6-12+8=1$. Therefore, the assertion follows.\\
(1)$\Rightarrow$(2).
A necessary condition for $R/I(G)$ to be Gorenstein is that its $h$-vector has to be symmetric (see \cite[Corollary 5.3.10]{Vi0}). From Proposition \ref{hvebu}, by requiring that 
\[
h_1=h_2,
\]
we obtain
\[
n(s-3)=-6.
\]
The pairs of integers with $n>2$ that satisfy the equation above are
\[
(3,1), (6,2).
\]
The first pair is not admissible because $\Delta(G)$ will be the simplex on $3$ vertices and $G$ the empty graph. So the $h$-vector has $h_{1}=h_{2}$ if $n=6$ and $s=2$. So the only candidates for $G$ are 
\[
C_{6}(1), \ \ C_{6}(2), \ \ C_{6}(3).
\]
In the first case $\Delta (G)$ is not pure, hence $R/I(G)$ cannot be Cohen-Macaulay. In the second case $\dim R/I(G)=2$. The assertion follows.
\end{proof}

\end{document}